\documentclass[11p,reqno]{amsart}
\usepackage{amsmath}
\usepackage{amsfonts}
\usepackage{amssymb}
\usepackage{graphicx}
\usepackage{color}
\newtheorem{theorem}{Theorem}[section]\newtheorem{thm}[theorem]{Theorem}
\newtheorem*{theorem*}{Theorem}
\newtheorem{lemma}{Lemma}[section]
\newtheorem{corollary}[theorem]{Corollary}
\newtheorem{proposition}{Proposition}[section]
\newtheorem{prop}{Proposition}[section]
\newtheorem{remark}[theorem]{Remark}

\def \b {\beta}

\def\Ric{\text{Ric}}

\def\bbar{{\bar{\beta}}}
\def\a{\alpha}
\def\l{\lambda}

\def\g{\gamma}

\def\abb{\alpha{\bar{\beta}}}
\def\gbd{\gamma{\bar{\delta}}}

\def\Ric{\operatorname{Ric}}
\def\Rm{\operatorname{Rm}}
\def\I{\operatorname{I}}
\def\Scal{\operatorname{Scal}}

\def\Im{\operatorname{Im}}

\def\I{\operatorname{I}}

\numberwithin{equation}{section}

\begin{document}

\title[Shrinkers and ancient solutions]{K\"ahler-Ricci Shrinkers and Ancient Solutions with Nonnegative Orthogonal Bisectional Curvature}

\author{Xiaolong Li}
\address{Department of Mathematics, University of California, Irvine, Irvine, CA 92697, USA}
\email{xiaolol1@uci.edu}

\author{Lei Ni}\thanks{The research is partially supported by  ``Capacity Building for Sci-Tech Innovation-Fundamental Research Funds".   }
\address{Department of Mathematics, University of California, San Diego, La Jolla, CA 92093, USA}
\email{lni@math.ucsd.edu}


\maketitle

\begin{abstract}
    In this paper we prove  classification results for gradient shrinking Ricci solitons under two invariant conditions, namely nonnegative orthogonal bisectional curvature and weakly PIC$\mbox{}_1$,  without any curvature bound.  New results on ancient solutions for the Ricci and K\"ahler-Ricci flow are also   obtained. The main new feature is that no curvature upper bound is assumed.
\end{abstract}


\section{Introduction}
Let $M^n$ be a K\"ahler manifold and $R$ denotes the curvature tensor.
The {\it orthogonal bisectional curvature} ($B^{\perp}$ for short) is defined for a pair of $X, Y\in T'_x M$ with $\langle X, \overline{Y}\rangle =0$ as $R(X, \bar{X}, Y, \bar{Y})$.  This curvature  arises naturally in the Bochner formula involving $(1,1)$-forms on a K\"ahler manifold \cite{HSW, Gold}

For the compact manifolds, there exist classification results \cite{XChen, GuZhang} (cf. \cite{Wilking} for alternative arguments) under $B^\perp>0$ and $B^\perp\ge 0$ (abbreviated as NOB) condition. When the manifold is compact and $B^\perp>0$, the manifold has to be biholomorphic to the complex projective space $\mathbb{P}^n$ (\cite{GuZhang,Wilking}). However a  complete classification of compact K\"ahler manifolds with $B^\perp\ge0$ is still hinged upon  the understanding of noncompact K\"ahler manifolds with $B^\perp\ge 0$  (cf. \cite{GuZhang} Theorem 1.3 part (2)).  On the other hand  even under the stronger condition of positive bisectional curvature it is still unknown  whether or not  such a complete noncompact K\"ahler manifold is Stein except some special cases \cite{NT}.   Hence understanding the structure of  K\"ahler manifolds with NOB is an interesting area of research.
In view of the examples constructed in \cite{HT}, \cite{NNiu} and \cite{NZ} the orthogonal bisectional curvature $B^\perp$ is completely independent of the holomorphic sectional curvature, or the Ricci curvature. Recently there is a joint work of second author  \cite{NNiu} proving  a Liouville theorem for plurisubharmonic functions, which complements a recent result of Liu \cite{GangLiu},  and a  gap theorem in on K\"ahler manifolds with $B^\perp\ge 0$ and $\Ric\ge0$. A comparison theorem of  the orthogonal complex Hessian was also obtained recently in \cite{NZ} for K\"ahler manifolds with $B^\perp\ge 0$.

The $B^\perp\ge 0$ condition is also related to K\"ahler-Ricci flow. First $B^\perp\ge 0$ is a K\"ahler analogue of the nonnegative isotropic curvature (please see \cite{Wilking} for this connection via the Lie algebraic point of view, and Section 3 of this paper for the definition of non-negativity of the isotropic curvature). More importantly $B^\perp\ge0$ is also invariant under the K\"ahler-Ricci flow (cf.  \cite{XChen}, \cite{GuZhang}).  In complex dimension two, $B^\perp\ge 0$ is equivalent to isotropic curvature being nonnegative.
However, as it was pointed out in \cite{GuZhang} that when the complex dimension is at least three, $B^\perp\ge 0$ is a much weaker condition than the isotropic curvature being nonnegative.
Gradient shrinking K\"ahler-Ricci (Ricci) solitons (abbreviated as shrinkers) naturally arises in the study of K\"ahler-Ricci (Ricci) flow as the singularity models. A K\"ahler-Ricci shrinker is a triple $(M, g, f)$ consisting of a complete K\"ahler manifold $(M, g)$ together with a potential function $f$ such that the Ricci curvature, the Hessian of $f$, and metric tensor $g$ satisfy that
$R_{\alpha\bar{\beta}}+f_{\alpha\bar{\beta}}-g_{\alpha\bar{\beta}}=0$ and $f_{\alpha \b}=0$. The soliton structure is a generalization of Einstein (K\"ahler-Einstein) metrics.
 Classification results on shrinking K\"ahler-Ricci solitons are  important/useful  to understand  the flow. We
 first prove the following theorem in Section 2.

\begin{thm}\label{thm:1}
Let $(M^n,g,f)$ be a  complete gradient shrinking K\"ahler-Ricci soliton. Suppose $M$ has $B^{\perp} \ge0$ and its universal cover does not contain a factor of $\mathbb{C}$. Then $M$ is compact. \end{thm}

Clearly one can not expect such a result  for general K\"ahler manifolds. For example, in \cite{NZ} a complete unitary symmetric metric was constructed on $\mathbb{C}^n$ with $B^\perp>0$ and $\Ric>0$. Apart from the motivation from the study of complex structure of K\"ahler manifolds with $B^\perp\ge 0$, the above result is  motivated by a recent work of Munteanu-Wang \cite{MW}, where a similar statement was proved for gradient shrinking Ricci soltions under the assumption that {\it  the sectional curvature is nonnegative and Ricci is positive.} In comparison, our result does not make any assumption on the Ricci curvature. In fact we prove that the Ricci curvature is nonnegative as a consequence of soliton equation even though the Ricci curvature a priori has nothing to do with $B^\perp$. Since in general $B^\perp\ge 0$ is a condition completely independent of $\Ric$, nor the holomorphic sectional curvature, not mentioning the sectional curvature, one can not derive our result from \cite{MW}.

Theorem \ref{thm:1} implies a  complete classification of K\"ahler-Ricci shrinkers with $B^\perp\ge0$  as a corollary.
\begin{thm}\label{thm:12}
Let $(M^n,g,f)$ be a complete gradient shrinking K\"ahler-Ricci soliton with $B^{\perp} \geq 0$.
Then the universal cover $\tilde{M}$ of $M$ split isometrically-holomorphically as $N_1 \times N_2 \times \cdots \times N_l \times \mathbb{C}^k$, where $N_i$ are compact irreducible Hermitian symmetric spaces.
\end{thm}
A classification {\it for gradient shrinking Ricci solitons with nonnegative curvature operator } was obtained  in \cite{MW}.
There is also an earlier related result of the second author \cite{Ni-MRL} asserting the compactness under the assumption that the bisectional curvature is positive, and a classification of K\"ahler-Ricci shrinkers with nonnegative bisectional curvature. Theorem \ref{thm:12} is a generalization of  these two previous results in the K\"ahler category. Both the work of \cite{MW} and \cite{Ni-MRL}  were motivated by Perelman's result \cite{Perelman}  asserting that {\it any three-dimensional shrinking solitons with bounded positive sectional curvature must be compact}. This result of Perelman together with the work of Hamilton (as well as Hamilton-Ivey pinching) provides a complete classification of shrinkers in three dimensions. See also \cite{NW} for a proof of a generalization of this result of Perelman via a PDE approach. As the compactness results of \cite{MW} and \cite{Ni-MRL}, Theorem \ref{thm:1} provides another high-dimensional  generalization of  Perelman's above statement for the K\"ahler case.

The method employed in proving Theorem \ref{thm:12} can also be adapted to prove a classification result, Theorem \ref{thm:31},  for shrinkers with weakly  PIC$\mbox{}_1$ condition (hence also  gives a similar result for shrinkers with 2-nonnegative curvature operator since $2$-nonnegativity of the curvature operator is stronger than PIC$\mbox{}_1$). In fact what is proved  in Section 3 is a bit more general.   In Section 4, the method is extended further to prove that any shrinkers with weakly PIC must have $2$-nonnegative Ricci curvature. In dimension four, joint with K. Wang,  a classification result for  shrinkers under weakly PIC condition \cite{LNW} has been obtained. Hence in this paper, for the discussion of shrinkers of weakly PIC  $n=\dim_{\mathbb{R}}(M)\ge 5$ is assumed. Note that in a recent work \cite{BCW}, it was shown that for $n\ge 7$, $R$ being of  weakly PIC implies that the Ricci curvature is $3$-nonnegative. There have been many works on gradient shrinking solitons since \cite{Perelman}.  One can refer to \cite{MW} and the book \cite{Chowetc-4} for  some  comprehensive descriptions of gradient K\"ahler-Ricci (Ricci)  solitons and  some known  results on shrinkers other than the ones mentioned here.

In the later sections we extend some of results proved for shrinkers in the earlier sections to ancient solutions of the Ricci and K\"ahler-Ricci flow. In particular we show that
\begin{thm}
Assume that $(M, g(t))_{t\in (-\infty, 0)}$ is an ancient solution to the Ricci flow or K\"ahler-Ricci flow. Then

(i) If $B^\perp\ge 0$, then $(M, g(t))$ has nonnegative bisectional curvature;

(ii) If $(M, g(t))$ has weakly PIC$\mbox{}_1$, then it has nonnegative complex sectional curvature;

(iii) If $(M, g(t))$ has weakly PIC, then $\Ric$ is $2$-nonnegative.
\end{thm}
Note that the part (i) was known for compact manifolds \cite{Wilking}, and the  part (ii) of the above was proved under additional assumption of  bounded curvature  recently in \cite{BCW}.  The main feature of our   results is that no upper curvature bound is assumed. This new feature could be desirable in applications. To achieve  the result without curvature bound  we  apply PDE  arguments via  differential inequalities on various geometric quantities and the viscosity consideration, since
 the approach via the dynamic version of Hamilton's maximum principle reducing the PDE to an ODE by dropping the diffusion term, as done  in Section 1 of  \cite{BW} and Theorem 12.38 of \cite{Chowetc}, has limited effect  that a curvature bound assumption is always needed.

Finally as an application of part (i) of the above result,  we extend the recent important  result of Balmer-Cabezas-Rivas-Wilking on the Ricci flow under almost nonnegative curvature conditions to include the case of the K\"ahler-Ricci flow under the almost nonnegative orthogonal bisectional curvature.

\begin{theorem}\label{thm:anob}
 For any $n\ge 2, \ne 3 $ and $\nu_0$, there exist positive constants $C=C(n, \nu_0)$ and $\tau=\tau(n, \nu_0)$ such that if $(M, g)$ is K\"ahler manifold with bounded curvature, $\dim_{\mathbb{C}}(M)=n$,
 $$Vol_g(B_g(p, 1))\ge \nu_0, \forall p\in M,$$
 and $\Rm+\epsilon \operatorname{id}$ has NOB for some $\epsilon \in [0, 1]$, then the K\"ahler-Ricci flow exists on $[0, \tau]$ with $\Rm_{g(t)}+C \epsilon \operatorname{id}$ has NOB and $|\Rm|\le \frac{C}{t}$ for all $t\in (0, \tau]$.
\end{theorem}
As a consequence of the above one can have a similar result as Corollary 3 of \cite{BCW}. Namely {\it for given $D>0$, $v_0>0$, there exists an $\epsilon=\epsilon(D, v_0, n)$ such that if a K\"ahler manifold $(M^n, g)$ satisfies that $Vol(M)\ge v_0$, $Diam(M)\le D$, and $\Rm+\epsilon \operatorname{id}$ has NOB, then $M$ admits a K\"ahler metric with NOB. In particular any such a simply-connected manifold with $b_2=1$ must be diffeomorphic to a Hermitian symmetric space of compact type.}  For bisectional curvature, the topological consequence of almost nonnegative K\"ahler manifolds was obtained earlier in \cite{Fang} under additional assumption of a uniform  bound of the sectional curvature. The Corollary 3 of \cite{BCW} has a similar topological consequence for K\"ahler manifolds with almost nonnegative bisectional curvature satisfying the same diameter and volume conditions.
As consequences of the classifications of shrinkers we also derive some classification results on closed
type-I noncollapsing ancient solutions.

\section{Proof of Theorem \ref{thm:1}}

We refer the reader to \cite{Chowetc-4} and Munteanu-Wang's paper \cite{MW} for basic equations concerning the gradient shrinking solitons.  We make the normalization on $f$ so that
$$ \Scal +|\nabla f|^2 =f,$$
where $\Scal$ denotes the scalar curvature.
Let $\Delta_f (\cdot) =\Delta (\cdot)  -\langle \nabla (\cdot), \nabla f\rangle$. The key of the proof is the following result.

\begin{prop}\label{prop:11}
Let $(M,g,f)$ be a complete shrinking K\"ahler-Ricci soliton with $B^{\perp}\geq 0$. Let $\lambda(x)$ be the minimum eigenvalue of the Ricci tensor at $x$. Then
\begin{equation}\label{eq:21}
    \Delta_f \lambda  \leq \lambda -\lambda^2,
\end{equation}
in the barrier or viscosity sense.
\end{prop}

\begin{proof}
Recall \cite{Hamilton4} that on a K\"ahler-Ricci shrinker, the Ricci tensor satisfies
$$\Delta_f R_{\a \bbar} =R_{\a \bbar} -R_{\abb \gbd } R_{\delta \bar\g}. $$
For any $p\in M$, choosing unitary frame $\{e_{\a}\}_{\a=1}^{n}$ at $p$ such that $R_{\abb} =\l_{\a} \delta_{\a \b}$ with $\l_1 \leq \l_2 \leq \cdots \leq \l_n$ gives at $p$,
\begin{eqnarray*}
\Delta_f R_{1\bar1} &\leq&  R_{1\bar1} -R_{1\bar1 \g \bar\g} R_{\g \bar\g} \\
&=& R_{1\bar1} -R_{1\bar1 1 \bar1} R_{1\bar 1} - \sum_{\g=2}^n R_{1\bar1 \g \bar\g}R_{\g \bar\g} \\
&\leq & R_{1\bar1} -R_{1\bar1 1 \bar1} R_{1\bar 1} - \sum_{\g=2}^n R_{1\bar1 \g \bar\g} R_{1 \bar 1} \\
&=& R_{1\bar 1} -(R_{1\bar1})^2.
\end{eqnarray*}
Applying a barrier argument in case $R_{1\bar1}$ is not smooth we have the result.
\end{proof}

The next step is to use the equation satisfied by $\l$, namely (\ref{eq:21}), to show that $\l \geq 0$. The proof given below for a shrinking soliton follows a localization technique  which of course has a long root in the study of PDEs.  Recent adaptation of this technique can be found in \cite{BChen, Yokota} etc. An exposition of it can also be found in the book \cite{Chowetc-4} (Chapter 27, Theorem 27.2). For shrinking solitons, our proof below is a bit cleaner.

\begin{prop}
Let $(M,g,f)$ be a complete gradient shrinking K\"ahler-Ricci soliton with $B^{\perp}\geq 0$. Let $\lambda(x)$ be minimum eigenvalue of Ricci tensor at $x$. Then $\lambda \geq 0$ on $M$. In particular, $M$ has $\Ric\geq 0$.
\end{prop}

\begin{proof}
We may assume $M$ is noncompact since the result follows immediately from the maximum principle for the compact case.   We prove the result by contradiction. Assume that $\sup_{K}\lambda\le -a<0$ for a compact subset $K$ sufficiently large with $a$ being a positive constant.
The potential function $f$ is an exhaustion function with the estimate \cite{CaoZhu} in terms of distance function $r(x)$ to a fixed point:
\begin{equation}\label{eq:22}
\frac{1}{4}\left(r(x)-c_1\right)^2-C_3 \le f(x)\le \frac{1}{4}\left(r(x)+c_2\right)^2.
\end{equation}
Hence for sufficiently large $R$, $D(R)=\{x\, |\, f(x)\le R\}$ contains $K$.

Let $\eta:[0,\infty) \to [0,1]$  be a smooth nonincreasing cutoff function with $\eta(s)=1$ for $0\leq s \leq 1$, $\eta(s)=0$ for $s \geq 2 $, and $|\eta''|+2\frac{(\eta')^2}{\eta} \leq C$, where $C$ is a universal constant. Let $\psi=\eta\left(\frac{f}{R}\right)$.
Consider $Q=\psi(x)\cdot \lambda(x)$. For $R$ large enough,  $\min_{D(R)}Q\le -a$.  We shall derive a contradiction by applying the maximum principle with a cut-off.

First direct calculation shows that
$$
\Delta_f \psi=\frac{\eta''}{R^2}|\nabla f|^2 +\frac{\eta'}{R}\Delta_f f; \quad |\nabla \psi|^2=\frac{(\eta')^2}{R^2}|\nabla f|^2.
$$
Now applying  the maximum principle at $x_0$ where the minimum of $Q$ (is $\le -a$ as we have seen above) is attained, we have at $x_0$ that
\begin{eqnarray*}
0&\le& \Delta_f Q =  \lambda \Delta_f  \psi  +\psi \Delta_f \lambda -2\frac{|\nabla \psi|^2}{\psi}\lambda\\
&\le& \lambda \left(\frac{\eta''}{R^2}|\nabla f|^2 +\frac{\eta'}{R}\Delta_f f\right)+\psi(\lambda-\lambda^2)-2\lambda \frac{(\eta')^2}{R^2}|\nabla f|^2\\
&\le& (-C\lambda)\frac{|\nabla f|^2}{R^2}+ (\eta'\lambda)\frac{\frac{n}{2}-f}{R} +\psi(\lambda-\lambda^2)\\
&\le&\frac{-C\lambda}{R}+(-\lambda) \frac{\frac{n}{2}+\|f^{-}\|_\infty}{R} +\psi(\lambda-\lambda^2).
\end{eqnarray*}
Here $f^{-}$ denotes the negative part of $f$.
Multiplying both sides of the above estimate  by $\psi$ we have that
\begin{equation}\label{eq:23}
0\le Q\left(\frac{-C'}{R}+\psi -Q\right)
\end{equation}
with $C'$ independent of $R$. Letting $R\to \infty$, this implies that $-a\ge \frac{C'}{R}\to 0$, a contradiction.
\end{proof}

Now evoking the strong maximum principle proved for general $(1,1)$-form $\eta\ge 0$ in \cite{NNiu}, under the assumption of $B^\perp\ge 0$, the kernel of $\Ric$ is invariant under the  parallel transport.  Hence the distribution (of dimension $k$) will split off a factor of $\mathbb{C}^k$ isometrically. If we also apply the argument of Munteanu-Wang we can have the following corollary, from which Theorem \ref{thm:1} follows.

\begin{corollary}
Let $(M^n, g, f)$ be complete K\"ahler-Ricci shrinker with $B^\perp\ge 0$. Then the finite covering universal cover $\tilde{M}=M_1^{n-k}\times \mathbb{C}^k$ with $M_1$ being a K\"ahler-Ricci shrinker of $\Ric>0$ and compact.
\end{corollary}
\begin{proof}
The first part follows from that $\Ric\ge 0$ and the strong maximum principle for $\Ric$ under the condition $B^\perp\ge0$. After the splitting, the non-Euclidean factor  $M_1$ must have $\Ric>0$. To show that $M_1$ is compact, we adapt the second part of Munteanu-Wang's argument to the K\"ahler setting. First of all  $\Delta_f \lambda \le \lambda -\lambda^2$ is all one needs to apply the result of Chow-Lu-Yang \cite{CLY} and obtain the lower estimate
$$
\Ric\ge \frac{b}{f}.
$$
Here the constant $0<b\le 1$. For this estimate one  can also use the argument in \cite{MW}.

Now we can repeat the argument of Munteanu-Wang \cite{MW} on pages 503-504, only observing that
$$
-\Ric\left(\frac{\nabla f}{|\nabla f|}, \frac{\nabla f}{|\nabla f|}\right) =-\Ric(E_1, \bar{E_1})\le -\frac{b}{f}
$$
where $E_1=\frac{1}{\sqrt{2}}\left(\frac{\nabla f}{|\nabla f|}-\sqrt{-1}J\left(\frac{\nabla f}{|\nabla f|}\right)\right)$.
This implies the lower estimate of scalar curvature $S\ge  b\log f(x)$, which then induces a contradiction due to the upper average estimate of the scalar curvature.
\end{proof}

A result of \cite{Wilking} (Section 4) asserts that the K\"ahler-Ricci flow evolves the $B^\perp\ge 0$ cone  into the cone of curvatures of nonnegative bisectional curvature on any compact K\"ahler manifold as $t\to$ the singular time. In particular, any compact ancient solution with $B^\perp\ge 0$ must have nonnegative bisectional curvature, which implies that $M_1$ admits nonnegative bisectional curvature since the shrinker is a singularity model.  Hence a complete classification with $B^{\perp} \geq 0$ can be obtained by appealing to an earlier result of the second author \cite{Ni-MRL}.

\begin{thm}\label{thm:22}
Let $(M^n,g,f)$ be a complete gradient shrinking K\"ahler-Ricci soliton with $B^{\perp} \geq 0$.
Then the universal cover $\tilde{M}$ of $M$ split isometrically-holomorphically as $N_1 \times N_2 \times \cdots \times N_l \times \mathbb{C}^k$, with each $N_i$ being a compact irreducible Hermitian symmetric space.
\end{thm}

The result on the compact factors being Hermitian symmetric spaces can also be seen via the fact that if a $\mathbb{C}$ factor exists, then by the observation of \cite{Wilking} the compact factor has nonnegative bisectional curvature. Otherwise, $\tilde{M}$ itself is compact, thus having nonnegative bisectional curvature by \cite{GuZhang} as well as \cite{Wilking}.

\section{Weakly PIC$\mbox{}_1$ Shrinkers}
Here we prove a similar result for gradient shrinking Ricci solitons with weakly  PIC$\mbox{}_1$.  Recall that a shrinker is a triple $(M^n, g, f)$ satisfying that $R_{ij}+f_{ij}-\frac{1}{2}f_{ij}=0$. We say that $(M^n,g)$ has PIC$\mbox{}_1$ if for any $p \in M$, for any orthonormal four-frame $\{e_1, e_2,e_3,e_4\}$ in $T_pM$ and any $\lambda \in [0,1]$,
\begin{equation}\label{eq:31} R_{1313}+\lambda^2 R_{1414}+R_{2323}+\lambda^2 R_{2424} -2\lambda R_{1234} >0.\end{equation}
Here $n=\dim_{\mathbb{R}}(M)$, unlike in the previous section where $n=\dim_{\mathbb{C}}(M)$ of a complex manifold.
We say $(M,g)$ has weakly PIC$\mbox{}_1$ if $``>0"$ in (\ref{eq:31}) is replaced by $``\ge 0"$. We say that $R$ has PIC if (\ref{eq:31}) holds only for $\lambda=1$. The PIC condition was first proven to be invariant under the Ricci flow in \cite{Ng} and \cite{BS}. The weakly PIC$\mbox{}_1$ condition was first introduced by Brendle-Schoen \cite{BS}.
 It was proved later in \cite{Brendle} that the Ricci flow evolves a compact manifold with  a PIC$\mbox{}_1$  metric into a round spherical metric. Hence any compact shrinker with PIC$\mbox{}_1$ metric must be the round sphere or its quotient. Our focus is to classify all the shrinkers with weakly PIC$\mbox{}_1$ without any curvature bound assumption. Since  (\ref{eq:31}) or weakly PIC$\mbox{}_1$  does not imply the nonnegativity of sectional curvature, the Munteanu-Wang result \cite{MW} can not be applied  directly. Because of that the shrinkers with weakly PIC in dimension four has been classified by works of  \cite{NW2, LNW}, we assume  $n\ge 5$ for this discussion.

\begin{lemma}\label{lem:31}
Let $R$ be an algebraic curvature operator.    \\
(1) If $R$ is of  weakly PIC$\mbox{}_1$, then for all orthonormal three-frame $\{e_1, e_2,e_3\}$, we have
\begin{equation}\label{eq:wpic}R_{1313}+R_{2323} \geq 0.\end{equation}
In particular, $R$ has nonnegative Ricci curvature. If (\ref{eq:wpic}) holds $``>"$, it implies $\Ric>0$.\\
(2) Assume $n= 5$ and $R$ is weakly PIC$\mbox{}_1$, or $n\ge 6$ and $R$ is weakly PIC, then
\begin{equation}\label{eq:32}\Scal-2R_{nn}-2R_{11}= -2R_{1n1n} +\sum_{k,l=2}^{n-1} R_{klkl} \geq -2R_{1n1n}.\end{equation}
\end{lemma}

\begin{proof}
For part (1), choose $\l=0$ in \eqref{eq:31}.  For part (2), observe that by (1),
$\sum_{k,l=2}^{n-1} R_{klkl}\geq 0.$ \end{proof}

\begin{thm}\label{thm:31}
Let $(M^n, g, f)$ be a complete gradient shrinking Ricci  soliton.\\
(i) Suppose that $M$ has $\Ric>0$ and weakly PIC$\mbox{}_1$, or slightly weak condition (\ref{eq:wpic}). Then $M$ must be compact.  In particular, any shrinker with PIC$\mbox{}_1$ must be compact, hence isometric to $\mathbb{S}^n$ or its quotient.  \\
(ii) Suppose $M$ has weakly PIC$\mbox{}_1$ (or has 2-nonnegative curvature operator). Then the universal cover $\tilde{M}$ splits isometrically as $N_1\times N_2 \times \mathbb{R}^k$, where $N_1$ is a product of irreducible compact Hermitian symmetric spaces and $N_2$ being the product of irreducible compact Riemannian symmetric spaces. \end{thm}
\begin{proof}
Let $\lambda(x)$ be the minimum eigenvalue of the Ricci tensor at $x$. By Lemma \ref{lem:31}, weakly PIC$\mbox{}_1$ implies $\lambda \geq  0$ on $M$. The key is to show that $\lambda$ satisfies
$$\Delta_f \lambda \leq \lambda.$$ Once this holds we can proceed as in Munteanu-Wang \cite{MW} to conclude the compactness for part (i) if $\Ric$ is assumed to be positive. For PIC$\mbox{}_1$ shrinker, the first part implies that it must be compact. Then Brendle's result implies that it must be  spherical.

Choose orthonormal frame $\{e_1, \cdots, e_n\}$ at $p$ such that $\Ric(e_i,e_j)=\lambda_i \delta_{ij}$ with $\lambda_1 \leq \lambda_2 \leq \cdots \leq \lambda_n$.
Using $\Delta_f R_{ij} =R_{ij} -2R_{ikjl}R_{kl}$, we obtain
\begin{eqnarray*}
\Delta_f R_{11} \leq R_{11} -2\sum_{k=2}^n R_{1k1k}R_{kk}.
\end{eqnarray*}
We  show below that $\sum_{k=2}^n R_{1k1k}R_{kk}\geq 0$ under weakly PIC$\mbox{}_1$  condition. \\
Case 1: $R_{1k1k} \geq 0$ for all $2\leq k \leq n$. There is nothing to prove. \\
Case 2: $R_{1k1k} <0$ for some $2\leq k \leq n-1$.
Since for  $j\neq k$, we have $R_{1j1j}+R_{1k1k} \geq 0$ by Lemma \ref{lem:31}, for  $j\ne k$, $R_{1j1j}\ge 0$ .   Let $m=k+1$,
$$\sum_{j=2}^n R_{1j1j}R_{jj} \geq  R_{1k1k}R_{kk}+R_{1m1m}R_{mm} \geq (R_{1k1k}+R_{1m1m})R_{mm} \geq 0.$$
Case 3: $R_{1n1n}<0$. The weakly PIC$\mbox{}_1$ condition implies $R_{1j1j}+R_{1n1n}\geq 0$ for all $2\leq j \leq n-1$. We can estimate, using part (2) of Lemma \ref{lem:31},
\begin{eqnarray*}
\sum_{k=2}^n R_{1k1k}R_{kk} &=&
\sum_{k=2}^{n-1} R_{1k1k}R_{kk} + R_{1n1n}R_{nn} \\
&\geq&  -\sum_{k=2}^{n-1} R_{1n1n}R_{kk}+R_{1n1n}R_{nn}
= -R_{1n1n} \left(\sum_{k=2}^{n-1}R_{kk} -R_{nn} \right)  \\
&=& -R_{1n1n} \left(\Scal-R_{11} -2R_{nn} \right) \\
&\geq &  -R_{1n1n} \left(R_{11} -R_{1n1n} \right)\geq 0
\end{eqnarray*}
Thus we have proved that  $\sum_{k=2}^n R_{1k1k}R_{kk}\geq 0$ and
$\Delta_f \lambda \leq \lambda$.

For part (ii), we can apply the splitting result in \cite{Ni-MRL04}, which in turn models the argument in \cite{NT}.  Precisely the strong maximum principle, namely  Theorem 2.2 of \cite{Ni-MRL04} can be applied under the
condition
\begin{equation}\label{eq:33}
\sum_{k=j+1}^nR_{1k1k}R_{kk}\ge 0
\end{equation}
if $R_{11}=R_{22}=\cdots=R_{jj}=0$. One can refer pages 483-484 of \cite{NT} for details of the proof that the distribution associated with the kernel of $\Ric$ is invariant under the parallel transport.
It then follows  that the universal cover $\tilde{M}$ splits as $M_1\times \mathbb{R}^k$. The Euclidean factor is obtained from the fact that the kernel of $\Ric$ is invariant under parallel transport and De Rham's theorem.  The factor $M_1$ must have positive Ricci, hence must be compact by the part (i).
The above proof of part (i) can be adapted to show (\ref{eq:33}) verbatim.

It is not hard to see that the  two-nonnegativity of the curvature operator implies (\ref{eq:wpic}). Hence the proof applies to that case as well.
\end{proof}
\begin{remark}
The above argument works under the condition: for all orthonormal three-frame $\{e_1,e_2,e_3\}$,  $$R_{1212}+R_{2323}\geq 0.$$
This condition is slightly weaker than PIC$\mbox{}_1$.  It was shown in \cite{BCW} that weakly PIC$\mbox{}_1$ ancient solution with bounded curvature must have nonnegative complex sectional curvature. Hence if the curvature is assume to be bounded, Theorem \ref{thm:31} is a consequence of Munteanu-Wang's result. Our result has the advantage that it applies to a weaker condition (\ref{eq:wpic}), and does not assume any curvature bound.
\end{remark}

\section{Weakly PIC Shrinkers}
Here we show a partial result towards understanding the shrinkers with weakly PIC. We assume that $\dim\ge 5$ since  the four-dimensional shrinkers with weakly PIC have been understood \cite{LNW}.
Algebraically, weakly PIC condition immediately implies that the  Ricci curvature is 4-nonnegative. If $n \geq 7$, weakly PIC implies Ricci is 3-nonnegative (see for example \cite[pages 10-11]{BCW}). By adapting arguments of the pervious two sections we show here that for shrinkers with weakly PIC, the Ricci curvature is in fact $2$-nonnegative.

\begin{prop}
Let $(M^n ,g, f)$ be a complete gradient shrinking Ricci soliton with weakly PIC. Then the Ricci curvature is 2-nonnegative.
\end{prop}

\begin{proof}
Apply the cut-off argument of Section 2  to the function $\lambda(x)$ in Proposition \ref{prop:42}. Then we can conclude that $M$ has 2-nonnegative Ricci curvature.
\end{proof}
\begin{remark}
The proof remains valid if one replaces weakly PIC by the weaker condition that for all orthonormal four-frame $\{e_1, e_2, e_3, e_4\}$,
\begin{equation}\label{eq:40}R_{1313}+R_{1414}+R_{2323}+R_{2424} \geq 0.
\end{equation}
The result above can also be extended to ancient solutions with weakly PIC. There is a strong maximum principle associated with the proposition below.\end{remark}

\begin{prop} \label{prop:42}
Let $(M^n,g,f)$ be a complete gradient shrinking Ricci soliton with weakly PIC. Denote by $\lambda_1(x) \leq \lambda_2(x) \leq \cdots \leq \lambda_n(x)$ the eigenvalues of the Ricci tensor at $x$.
Then the function $\lambda(x)=\min\{\lambda_1(x) +\lambda_2(x), 0\}$ satisfies
\begin{equation}\label{eq:41}
    \Delta_f \lambda  \leq \lambda -\lambda^2,
\end{equation}
in the barrier or viscosity sense.
\end{prop}

\begin{proof}
Recall that on a shrinker, $\Delta_f R_{ij}=R_{ij}-2R_{ikjl}R_{kl}$. Choose orthonormal frame $\{e_i\}_{i=1}^n$ at $p$ such that $\Ric(e_i,e_j)=\lambda_i\delta_{ij}$ and $\lambda_1 \leq \lambda_2 \leq \cdots \leq \lambda_n$.
The result follows from the estimate below:
\begin{eqnarray*}
\Delta_f (R_{11}+R_{22})
&\leq&  R_{11}+R_{22}
-2 \sum_{k=1}^n ( R_{1k1k}+R_{2k2k} )R_{kk} \\
&=& R_{11}+R_{22} -2R_{1212}(R_{11}+R_{22})
- 2\sum_{k=3}^n ( R_{1k1k}+R_{2k2k} ) R_{kk} \\
&=& R_{11}+R_{22} -(R_{11}+R_{22})^2
- \sum_{k=3}^n ( R_{1k1k}+R_{2k2k} )  (2R_{kk}-R_{11}-R_{22} ) \\
&\leq & R_{11}+R_{22} -(R_{11}+R_{22})^2
\end{eqnarray*}
where we have used $2R_{1212}=R_{11}+R_{22}-\sum_{k=3}^n (R_{1k1k}+R_{2k2k})$ and Lemma \ref{lem:41} to be proved next.
\end{proof}

\begin{lemma}\label{lem:41}
Let $R$ be an algebraic curvature operator with weakly PIC. Let $\{e_1, e_2, \cdots, e_n \}$ be an orthonormal frame such that $\Ric(R)$ is diagonal with $R_{11} \leq R_{22} \leq \cdots \leq R_{nn}$ being the eigenvalues of $\Ric(R)$. Suppose $R_{11}+R_{22} \leq 0$.
Then $$\sum_{k=3}^n ( R_{1k1k}+R_{2k2k} )  (2R_{kk}-R_{11}-R_{22} ) \geq 0.$$
\end{lemma}

\begin{proof}
Consider two cases: \\
Case A: $R_{1k1k}+R_{2k2k} \geq 0$ for all $3\leq k \leq n$. In this case, given $R_{kk}\ge \max\{R_{11}, R_{22}\}$
\begin{eqnarray*}
&& \sum_{k=3}^n ( R_{1k1k}+R_{2k2k} )  (2R_{kk}-R_{11}-R_{22} )\ge 0.
\end{eqnarray*}
Case B: $R_{1p1p}+R_{2p2p} < 0$ for some $3\leq p \leq n$. Since $R$ is weakly PIC, we have that for all $3\leq k \leq n, k\neq p$,
$$R_{1k1k} +R_{2k2k} \geq -(R_{1p1p}+R_{2p2p}) > 0.$$
Hence if $p<n$, let $m=p+1$ we have that
\begin{eqnarray*}&&  \sum_{k=3}^n( R_{1k1k}+R_{2k2k} )  (2R_{kk}-R_{11}-R_{22} ) \ge \left(R_{1p1p}+R_{2p2p}\right)(2R_{pp}-R_{11}-R_{22})\\&\,& \quad\hskip2cm +
\left(R_{1m1m}+R_{2m2m}\right)(2R_{mm}-R_{11}-R_{22})\\
&\, &  \quad\hskip2cm\ge (R_{1p1p}+R_{2p2p}+R_{1m1m}+R_{2m2m})(2R_{pp}-R_{11}-R_{22})\ge 0.
\end{eqnarray*}
If $p=n$ as before we have
\begin{eqnarray*}
&& \sum_{k=3}^{n-1} ( R_{1k1k}+R_{2k2k} )  (2R_{kk}-R_{11}-R_{22} ) + ( R_{1n1n}+R_{2n2n} )  (2R_{nn}-R_{11}-R_{22} )\\
&\ge& -( R_{1n1n}+R_{2n2n} ) \left(2\sum_{k=3}^{n-1}R_{kk}- 2R_{nn}-(n-4)(R_{11}+R_{22} )\right) \\
&=& -( R_{1n1n}+R_{2n2n} )\left(2\Scal-4R_{nn}-(n-2)(R_{11}+R_{22} )\right)\ge 0.
\end{eqnarray*}
In the last step we have used the assumption $R_{11}+R_{22} \leq 0$ and the fact that for $n\ge 5$ $\Scal-2R_{nn}=\sum_{k,l=1}^{n-1} R_{klkl} \geq 0$.
\end{proof}

\section{Ancient Solutions with $B^\perp \ge 0$}
In this section, for simplicity $n$ denotes the real dimension for a Riemannian manifold, and the complex dimension for a K\"ahler manifold. First, we extend the argument in the previous discussion to  show that

\begin{proposition}\label{prop-ob-ricci}
Any ancient solution of K\"ahler-Ricci flow with $B^\perp\ge 0$ must have $\Ric\ge 0$.
\end{proposition}
Again the virtue of the this result is that no curvature upper bound is assumed. Also as pointed out before such a result can not be true in general due to the examples constructed in \cite{NNiu}.   Below is a proof of this statement.

As in \cite{Hamilton4} we apply the Uhlenbeck's trick of gauge fixing by introducing the map $u:E\to T'M$ which is an identity at $t=-1$ (say the ancient solution is defined on $(-\infty, 0)$) satisfying the ODE
$$
\frac{\partial u^i_j}{\partial t}=\frac{1}{2}
g^{i\bar{k}}R_{k\bar{s}}u^s_j.$$ Then define the bundle metric $h(X, \bar{Y})=g(u(X), \overline{u(Y)})$. It is easy to check that $\frac{\partial}{\partial t}h=0$. One can  pull-back the complex structure (from $T_{\mathbb{C}}M=T'M\oplus T''M$ to $E\oplus \bar{E}$ via $u$),   the connection and the curvature to $E$ via $u$. The pull back curvature satisfies the PDE:
\begin{equation}\label{eq:51}
\frac{\partial R_{i\bar{j}k\bar{l}}}{\partial t}-\Delta R_{i\bar{j}k\bar{l}}=R_{i\bar{j}q\bar{p}}R_{p\bar{q}k\bar{l}}+R_{i\bar{l}q\bar{p}}R_{p\bar{q}k\bar{j}}-R_{i\bar{p}k\bar{q}}R_{p\bar{j}q\bar{l}}.
\end{equation}
Tracing it we have
\begin{equation}\label{eq:52}
\frac{\partial R_{i\bar{j}}}{\partial t}-\Delta R_{i\bar{j}}=R_{i\bar{j}q\bar{p}}R_{p\bar{q}}.
\end{equation}
The covariant derivative and Laplacian are computed with respect to the changing metric (along with the  induced Levi-Civita connection) on the manifold and the induced time-dependent connection on $E$.   The nonnegativity of $\Ric$ stay invariant under pulling back by $u$.

To prove the claimed result for the ancient solutions we first observe that
the argument of Proposition \ref{prop:11} implies the following lemma.

\begin{lemma} Let $(M, g(t))_{t\in (-\infty, 0)}$ be an ancient solution to the K\"ahler-Ricci flow with $B^\perp\ge 0$. Then the minimum of the Ricci curvature, denoted as  $\lambda$, satisfies in the barrier or viscosity sense, the partial differential inequality:
\begin{equation}\label{eq:53}
\frac{\partial \lambda}{\partial t}-\Delta \lambda \ge \lambda^2.
\end{equation}
\end{lemma}

For ancient solutions it is convenient to introduce a parameter $\tau:= -t$ and consider the $M\times(0, \infty)$. To prove our assertion on the Ricci curvature we also need a result of Perelman on the time dependent distance function (cf. Lemma 8.3 of  \cite{Perelman}).
\begin{lemma}[Perelman] \label{perelman}
(a) Assume that $\Ric(\cdot, \tau_0) \le (2n-1)K$ on the ball $B_{\tau_0}(x_0,  r_0)$. Then outside
of $B_{\tau_0} (x_0, r_0)$,
\begin{equation}\label{eq:54}
\left(\frac{\partial}{\partial \tau}  + \Delta_{g(\tau_0)}\right) d_{\tau_0}(\cdot, x_0) \le  (2n-1)\cdot\left(\frac{2}{3} Kr_0 + r^{-1}_
0\right).
\end{equation}
The inequality is understood in the barrier sense.

(b) Assume that $\Ric(\cdot, \tau_0)\le  (2n-1)K$ on the union of the balls $B_{\tau_0} (x_0, r_0)$
and $B_{\tau_0} (x_1, r_0)$. Then
\begin{equation}\label{eq:55}
\left.
\frac{d^+}{d\tau}
 d_{\tau_0}(x_0, x_1)\right|_{\tau=\tau_0}
\le  2(2n-1)\cdot \left(\frac{2}{3}
Kr_0 + r^{-1}_0\right).
\end{equation}
Here, $\frac{d^+ f}{d\tau}=\limsup_{\epsilon\to 0^+}\frac{ f(\tau+\epsilon)-f(\tau)}{\epsilon}$ denotes the upper Dini
derivative.
\end{lemma}
Now we  assume that $\lambda(x_0, \tau_0)<0$ and use (\ref{eq:53}) and (\ref{eq:54}) to derive a contradiction. First we need to construct an auxiliary function.

Let $\eta$ be a smooth nonincreasing function on the real line satisfying: (i) $\eta(s)=1 $ for $s\in (-\infty, \frac{1}{2}]$ and $\eta(s)=0$ for $s\in [1, \infty)$; (ii) $\eta''-\frac{2(\eta')^2}{\eta}\ge -A\sqrt{\eta}$ with $A>0$ being an absolute constant. This kind of function can be easily constructed and was also employed by Perelman in \cite{Perelman} (Chapter 10). By translating the time we may assume that $\tau_0=0$. We shall construct an auxiliary function $\psi$ which has compact support, and apply  the maximum principle to $Q:= \psi  \lambda$ to derive a contradiction.

First pick a $T_0$ such that $8|\lambda|^{-1}(x_0, 0)\le T_0$. Now we find a $r_0$ such that $\Ric\le \frac{2n-1}{r_0^2}$ for any $x\in B_\tau(x_0, r_0)$ and $\tau\in [0, T_0]$. This clearly can be done since as $r_0\to 0$, the upper bound expression  $\frac{2n-1}{r^2_0} \to \infty$. On the other hand  the Ricci has a fixed upper on a fixed compact subset $K\times[0, T_0]$, which contains $\overline{B_\tau(x_0, r_0)}\times [0, T_0]$. Now we choose a constant $B$ such that
$B^2\ge \frac{2 A}{|\lambda|(x_0, 0) \cdot r_0^2}$. Now let
$$\psi(x, \tau)=\eta\left(\frac{d_\tau(x, x_0)-\frac{5}{3}(2n-1) \frac{\tau}{r_0}}{B r_0}\right), \quad  Q:=\psi \lambda.$$
For any $\tau\ge 0$, $\psi$ has compact support in $B_\tau(x_0, Br_0+\frac{5}{3}(2n-1)\frac{\tau}{r_0})$. Let $Q(\tau)$ denotes the minimum of $Q(x, \tau)$ at the time slice $M\times \{\tau\}$. It is negative for $\tau$ close to $0$ (and stay negative as $\tau$ increases as shown below) and it is  attainted somewhere within finite distance away from $x_0$, which we denote as $x_\tau$. We shall derive the changing rate estimate of $Q(\tau)$.

Case 1): The point $x_\tau$ satisfies $d(x_\tau, x_0)\le r_0$, then by the construction $\psi(x_\tau)=1$ in the small neighborhood of $x_\tau$, hence we have that
\begin{equation}\label{eq:56}
\frac{d^+ }{d\tau} Q(\tau)=\frac{d^+ }{d\tau} \lambda(\tau)\le -\Delta \lambda- \lambda^2\le -Q^2(\tau).
\end{equation}
In the above we have used that $\Delta \lambda=\Delta Q \ge 0$ at the local minimum point.

Case 2): The negative minimum is attained at some point $x_\tau$ outside of the ball $B_\tau(x_0, r_0)$.  This allows us to  apply the distance comparison result in part (a) of Lemma \ref{perelman}, namely (\ref{eq:54}) to obtain the estimate:
\begin{eqnarray}
\left(\frac{\partial}{\partial \tau}+\Delta\right)\psi &=&\eta' \cdot \frac{  \left(\frac{\partial}{\partial \tau}+\Delta\right)d_\tau-\frac{5}{3}(2n-1)r_0^{-1}}{Br_0}+\eta'' \cdot \frac{1}{(Br_0)^2} \nonumber\\
&\ge&\eta'' \cdot \frac{1}{(Br_0)^2}. \label{eq:57}
\end{eqnarray}
By (\ref{eq:53})  that $\left(\frac{\partial}{\partial \tau}+\Delta\right) Q\le \lambda \left(\frac{\partial}{\partial \tau}+\Delta\right)\psi -\lambda^2 \psi +2\langle \nabla \psi, \nabla \lambda\rangle$.
And observe that at $(x_\tau, \tau)$, $\langle \nabla \psi, \nabla \lambda\rangle=-\frac{|\nabla \psi|^2}{\psi}\lambda$. Putting the three estimates above together we  have that  as long as $\lambda(x_\tau, \tau)\le 0,$
\begin{eqnarray*}
\frac{d^+ }{d\tau} Q(\tau) &\le& \left(\frac{\partial}{\partial \tau}+\Delta\right) Q\le \eta'' \cdot \frac{1}{(Br_0)^2} \lambda +\psi \left(\frac{\partial}{\partial \tau}+\Delta\right) \lambda-2\frac{ |\nabla \psi|^2}{\psi}\lambda\\
&\le& \left(\eta''-2\frac{(\eta')^2}{\eta}\right) \cdot \frac{\lambda }{(Br_0)^2} -\psi \lambda^2\\
&\le& -\frac{A}{(Br_0)^2} \sqrt{\psi}\lambda -\psi \lambda^2\le -\frac{1}{2}\psi \lambda^2 +\frac{1}{2}\left(\frac{A}{(Br_0)^2}\right)^2\\
&\le& -\frac{1}{2}Q^2 +\frac{1}{2}\left(\frac{A}{(Br_0)^2}\right)^2.
\end{eqnarray*}
By the choice of $B$ we have that $Q(\tau)$ is nonincreasing near $\tau=0$ and keep being so by the above estimate and (\ref{eq:56}), in views of the choices of $A$ and $B$ such that $|\lambda|(x_0, 0)\ge \frac{2A}{(Br_0)^2}$. Applying this back to the above estimate,  and combining the result with (\ref{eq:56}) we have the estimate
\begin{equation}\label{eq:58}
\frac{d^+ }{d\tau} Q(\tau)\le -\frac{1}{4}Q^2(\tau)
\end{equation}
which, after integration, implies the estimate
$$
Q(\tau)\le \frac{Q(0)}{1+Q(0)\frac{\tau}{4}} \to -\infty
$$
as $\tau\to T_1:=-\frac{4}{Q(0)}$,  which is clearly less than $T_0$. The contradiction then proves that $\lambda\ge 0$.

Combining the fact that $\Ric\ge 0$, the splitting result of \cite{NNiu}, together with the fact that $B^\perp\ge 0$ and $\Ric=0$ imply that the manifold is flat (see pages 8-9 of \cite{NNiu}) we have the following splitting theorem.

\begin{theorem}\label{OB-Split}
Let $(M, g(t))_{t\in (-\infty, 0)}$ be a nonflat ancient solution of the K\"ahler-Ricci flow with $B^\perp\ge 0$. Then the flow on its universal cover $\tilde{M}$ splits into $(M_1, g_1(t))\times (\mathbb{C}^k, g_{euc})$ such that $(M_1, g_1(t))$ has   $\Ric>0$, and  nonnegative bisectional curvature.
\end{theorem}

For the last statement we may appeal to the observation of \cite{Wilking} (page 226) stating that $M\times \mathbb{C}$ has $B^\perp\ge 0$ if and only if $M$ has nonnegative bisectional curvature. In the next section we show that in fact any ancient solution with $B^\perp\ge 0$ has nonnegative bisectional curvature.
The same argument of the above discussion proves the following result regarding the ancient solutions with weakly PIC.

\begin{proposition}
Let $(M, g(t))_{t\in (-\infty, 0)}$ be a nonflat ancient solution of the Ricci flow with weakly PIC. Then $\Ric$ is $2$-nonnegative.
\end{proposition}

To prove this, we simply need to observe that the argument of the proof of Lemma \ref{lem:41} implies that
\begin{equation}\label{eq:59}
(\partial_t -\Delta)(R_{11}+R_{22}) \ge (R_{11}+R_{22})^2 -K(R_{11}+R_{22})
\end{equation}
and $K$ can be taken to be zero if $R_{11}+R_{22}\le 0$, otherwise  $K$ is a constant, locally  depends on $\|R\|$.

This also allows us to evoke the strong maximum principle to conclude that if $R_{11}+R_{22}$ attains zero somewhere at $(x_0,t_0)$, then all the othonormal two-frame $\{e_1, e_2\}$ satisfying that $R_{11}+R_{22}=0$ must be invariant under the parallel transport. It is also easy to show that such $e_1, e_2$ must belong to a set of four-frame $\{e_1, e_2, e_3, e_4\}$ such that
$$
R_{1313}+R_{1414}+R_{2323}+R_{2424}-2R_{1234}=0.
$$
By \cite{BS} we have that $M$ can not have the holonomy group being $\mathsf{SO}(n)$ if it is simply-connected and irreducible. Thus we have

\begin{proposition}
Let $(M, g(t))_{t\in (-\infty, 0)}$ be a simply-connected  ancient solution of the Ricci flow with weakly PIC. Assume further that $M$ is irreducible. Then one of the following three holds: (i) $\Ric$ is $2$-positive, (ii) $(M, g)$ is a symmetric space, (iii) $(M, g)$ is an ancient solution to K\"ahler-Ricci flow with weakly $PIC$ and  $\Ric>0$.
\end{proposition}


\section{K\"ahler-Ricci flow under almost NOB condition}

First by combining the argument of the proof of Proposition \ref{prop-ob-ricci} and  a modification of the argument in \cite{BCW},  we strengthen Proposition \ref{prop-ob-ricci} to show that in fact the ancient solution with $B^\perp\ge 0$ has nonnegative bisectional curvature. This result is needed in extending a recent result of \cite{BCW} to K\"ahler manifolds with negative lower bound of $B^\perp$. We start with a lemma.

\begin{lemma}
Let $(M, g(t))_{t\in (-\infty, 0)}$ be a nonflat ancient solution of the K\"ahler-Ricci flow with $B^\perp\ge 0$. Let $u(x,t)$ be the function defined by
\begin{equation}\label{inf bisec}
   u(x,t)=\inf \{ \Rm_{(x,t)}(v,\bar{v}) | \mbox{ } v \in \Sigma \},
\end{equation}
where $\Sigma=\{v \in \mathfrak{gl}(n,\mathbb{C}) | \mbox{rank }(v)=1, \mbox{ and eigenvalues of norm }\le 1 \}$.
Then the function $\mu(x,t)=\min \{u(x,t), 0 \}$ satifies the partial differential inequality
\begin{equation}\label{mini bisec equ}
    \frac{\partial \mu}{\partial t}  -\Delta \mu \geq \mu^2
\end{equation}
in the barrier or viscosity sense.
\end{lemma}
\begin{proof}
By perturbation we may assume that $\Rm$ has $B^\perp>0$. By \cite{Wilking}, we know that $u(x,t) \geq 0$ if and only if $\Rm(x,t)$ has nonnegative bisectional curvature.
So it suffices to consider the case $\mu(x,t)<0$ as the inequality \eqref{mini bisec equ} follows from the proof of the fact that nonnegative bisectional curvature is preserved by K\"ahler-Ricci flow when $\mu(x,t)=0$.
In the rest of the proof, we fix a spacetime point $(x,t)$ and $\Rm(x,t)$ is  abbreviated as $\Rm$.
We claim the infimum in \eqref{inf bisec} is attained and finite. There are two cases, the first is that $u$ is $-\infty$. Then we may have $v_i\in \Sigma$ such that $\lim_{i\to \infty} \Rm(v_i,\bar{v}_i)= -\infty$. From this it is clear that $|v_i|\to \infty$. Moreover $\Rm(\frac{v_i}{|v_i|}, \overline{\frac{v_i}{|v_i|}} )$ converges to say $-a$ for some finite $a>0$ by passing to a subsequence. On the other hand, since $\frac{v_i}{|v_i|}\to v_\infty$ by passing to a subsequence, with $v_\infty$ being nilpotent, we also have $\Rm(v_\infty, \bar{v}_\infty)>0$, a contradiction.

If $0>u>-\infty$, let $v_i \in \Sigma$ be a minimizing sequence such that $\lim_{i\to \infty} \Rm(v_i,\bar{v}_i) =u(x,t)$. If $|v_i|$ remains bounded, then by compactness, we can pass to a subsequential limit $v_{\infty}$ with $\Rm(v_{\infty}, \bar{v}_{\infty}) =u(x,t)$. In case $|v_i| \to \infty$, let $w$ be a subsequential limit of the sequence $v_i/|v_i|$. Then $\Rm(w,\bar{w})=0$.  At the mean time $w$ has rank 1 and eigenvalues all zero. Thus $w^2=0$ and $\Rm(w,\bar{w})> 0$ by \cite{Wilking} in view of $B^{\perp} > 0$. The contradiction shows that the case $|v_i|\to \infty$ does not occur.

Let $v\in \Sigma$ be the matrix such that $u(x,t)=\Rm(v,\bar{v})$.
Now assume that $v=x\otimes \bar{y}$. Since the scaling $x \to \lambda x$ and $y \to \lambda^{-1}y$ for some $\lambda>0$ does not change $x\otimes y$, we may assume that $x$ has the norm of $|\langle x, \bar{y}\rangle|$, namely the norm of the eigenvalue of $v$. Now let $E=\frac{x}{|x|}$ be the unitary vector of $x$ direction and write $y= aE+y^{\perp}$, with $y^\perp\in \{E\}^\perp$. It is easy to see that $|x|\bar{a}=\langle x, \bar{y}\rangle$. Hence $a=e^{-\sqrt{-1}\theta}$ with $\theta$ being the argument of $\langle x, \bar{y}\rangle$. Write  $v=u+w=x\otimes \bar{a}\bar{E}+x\otimes \overline{y^{\perp}}$. Then $|u|=|\langle x, \bar{y}\rangle|$, the norm of the eigenvalue of $v$. Moreover $u+sw$ has rank one and has the eigenvalue $\bar{a} |x|$, which has the norm of $|u|$.

As in \cite{BCW}, the first variation gives that
$$
\Rm(u, \bar{w})+\Rm(w, \bar{u})+2\Rm (w, \bar{w})=0.
$$
This then implies that
$\mu(x,t)=\Rm(u, \bar{u})+\frac{1}{2}\Rm(u, \bar{w})+\frac{1}{2}\Rm(w, \bar{u})=\Re\left(\Rm(v, \bar{u})\right)$. Using that   $|u|\le 1$, we have
$$
\Rm^2(v, \bar{v})=\langle \Rm(v), \overline{\Rm(v)}\rangle =|\Rm(v)|^2 \ge |\Rm(v, \bar{u})|^2\ge \mu(x,t)^2.
$$
By \cite{Wilking} we also have $\Rm^{\#}(v, \bar{v})\ge0$. Hence we
have proved that
$ \frac{\partial \mu}{\partial t}  -\Delta \mu \geq \mu^2$ in the barrier sense.
\end{proof}

\begin{proposition}\label{prop:61}
Let $(M, g(t))_{t\in (-\infty, 0)}$ be a nonflat ancient solution of the K\"ahler-Ricci flow with $B^\perp\ge 0$. Then it has nonnegative bisectional curvature. Furthermore, if the curvature is bounded, then the volume growth is non-Euclidean, namely the asymptotic volume ratio
$\mathcal{V}(M, g(t))=0$.
\end{proposition}
\begin{proof}
By the same argument as in the proof of Proposition \ref{prop-ob-ricci}, we obtain that $u(x,t) \geq 0$ on $M\times (-\infty, 0)$. This proves that $(M, g(t))_{t\in (-\infty, 0)}$ has nonnegative bisectional curvature. The second statement now follows from Theorem 2 of  \cite{Ni-MRL}.
\end{proof}
Note that the above result generalizes Theorem 2 of \cite{Ni-MRL}.
The same argument shows that Lemma 4.2 in \cite{BCW} holds without the bounded curvature assumption.
\begin{proposition}\label{prop:62}
Let $(M, g(t))_{t\in (-\infty, 0)}$ be a nonflat ancient solution of the Ricci flow with weakly PIC$\mbox{}_1$. Then it has nonnegative complex sectional curvature.
\end{proposition}

Applying the argument of \cite{BCW}, in view of the above Proposition \ref{prop:61} we have the following result as the corollary.

\begin{theorem}\label{thm:61}
 For any $n\ge 2, \ne 3$ and $\nu_0$, there exist positive constants $C=C(n, \nu_0)$ and $\tau=\tau(n, \nu_0)$ such that if $(M, g)$ is an $n$-dimensional K\"ahler manifold with bounded curvature, and
 $$Vol_g(B_g(p, 1))\ge \nu_0, \forall p\in M,$$
 and $\Rm+\epsilon \operatorname{id}$ has NOB for some $\epsilon \in [0, 1]$, then K\"ahler-Ricci flow exists on $[0, \tau]$ with $\Rm_{g(t)}+C \epsilon \operatorname{id}$ has NOB and $|\Rm|\le \frac{C}{t}$ for all $t\in (0, \tau]$.
\end{theorem}
\begin{proof} As in Section 2.3 of \cite{BCW} define $\ell(x, t)$ as
$$
\ell(x, t):=\inf\{\alpha | (\Rm +\alpha \operatorname{id})(X\wedge\bar{Y}, \overline{X\wedge\bar{Y}})\ge 0,\, \forall X, Y\in T_x'M,  |X|=|Y|=1, \langle X, \bar{Y}\rangle =0\}.
$$
Here $\operatorname{id}$ is the curvature operator of $\mathbb{P}^n$, namely the one corresponding to $g_{i\bar{j}}g_{k\bar{l}}+g_{i\bar{l}}g_{k\bar{j}}$. In view of  the proof of Theorem 1 of \cite{BCW}, particularly Sections 3 and 4,  to prove the theorem, given Proposition \ref{prop:61} it suffices to show that
\begin{equation}\label{eq:6-key}
\left(\frac{\partial}{\partial t}  -\Delta\right)\ell \le \Scal \ell +C\ell^2.
\end{equation}
Here $C=C(n)$ is a dimensional constant. It is easy to see that if $\ell(x, t)>0$, then
$$
-\ell(x, t)=\inf\{ R_{X\bar{X}Y\bar{Y}}\, |\, \forall X, Y\in T_x'M,  |X|=|Y|=1, \langle X, \bar{Y}\rangle =0\}.
$$
Hence we can apply a similar computation as above to this setting. Pick a unitary frame $\{E_i\}$ such that $E_1=X$ and $E_2=Y$.  By (\ref{eq:51}) we have that
\begin{eqnarray*}
\left(\frac{\partial}{\partial t}  -\Delta\right) (-\ell)&=&\sum_{p, q=1}^2R_{1\bar{1}q\bar{p}}R_{p\bar{q}2\bar{2}}+|R_{1\bar{2}q\bar{p}}|^2-|R_{1\bar{p}2\bar{q}}|^2\\
&\,&+ \left(\sum_{p=1, 2; q\ge 3}+\sum_{q=1,2; p\ge 3}\right)
\left(R_{1\bar{1}q\bar{p}}R_{p\bar{q}2\bar{2}}+|R_{1\bar{2}q\bar{p}}|^2-|R_{1\bar{p}2\bar{q}}|^2\right)\\
&\,& + \sum_{p, q\ge 3}\left(R_{1\bar{1}q\bar{p}}R_{p\bar{q}2\bar{2}}+|R_{1\bar{2}q\bar{p}}|^2-|R_{1\bar{p}2\bar{q}}|^2\right).
\end{eqnarray*}
For the last term on the right above, the second variational consideration based on the fact that $R_{1\bar{1}2\bar{2}}$ attains the minimum of $B^\perp$ among all orthonormal two frame $\{X, Y\}$ (as in \cite{GuZhang}) shows that $\sum_{p, q\ge 3}R_{1\bar{1}q\bar{p}}R_{p\bar{q}2\bar{2}}-|R_{1\bar{p}2\bar{q}}|^2\ge 0$. Thus
$$
III=\sum_{p, q\ge 3}\left(R_{1\bar{1}q\bar{p}}R_{p\bar{q}2\bar{2}}+|R_{1\bar{2}q\bar{p}}|^2-|R_{1\bar{p}2\bar{q}}|^2\right) \ge \sum_{p, q\ge 3}|R_{1\bar{2}q\bar{p}}|^2 \ge 0.
$$
The second last term can be written as
\begin{eqnarray*}
II&=&\sum_{j\ge 3}R_{1\bar{1}1\bar{j}}R_{2\bar{2}j\bar{1}}+|R_{1\bar{2}1\bar{j}}|^2-|R_{1\bar{1}2\bar{j}}|^2+R_{1\bar{1}2\bar{j}}R_{2\bar{2}j\bar{2}}+|R_{1\bar{2}2\bar{j}}|^2-|R_{1\bar{2}2\bar{j}}|^2\\
&\,&+\sum_{j\ge 3}R_{1\bar{1}j\bar{1}}R_{2\bar{2}1\bar{j}}+|R_{1\bar{2}j\bar{1}}|^2-|R_{1\bar{j}2\bar{1}}|^2+R_{1\bar{1}j\bar{2}}R_{2\bar{2}2\bar{j}}+|R_{1\bar{2}j\bar{2}}|^2-|R_{1\bar{j}2\bar{2}}|^2.
\end{eqnarray*}
By considering the first variation of $f(\theta)=R(\cos \theta E_1+\sin\theta E_j, \overline{\cos \theta E_1+\sin\theta E_j}, E_2, \bar{E}_2)$ with the fact that $f(0)$ attains the minimum we have
$\Re R_{j\bar{1}2\bar{2}}=0$. Replacing $E_j$ by $\sqrt{-1} E_j$ we also have $\Im R_{j\bar{1}2\bar{2}} =0$. Hence $R_{j\bar{1}2\bar{2}}=0$. Similarly $R_{j\bar{2}1\bar{1}}=0$. Using these equations and symmetries of the curvature we have that
$$
II=\sum_{j\ge 3} |R_{1\bar{2}1\bar{j}}|^2+|R_{1\bar{2}j\bar{2}}|^2\ge 0.
$$
Applying a similar first variational consideration we also have $R_{1\bar{2}2\bar{2}}=R_{1\bar{1}1\bar{2}}$. Using this equation in the first sum of the right hand side of the equation for $\left(\frac{\partial}{\partial t}  -\Delta\right) (-\ell)$
\begin{eqnarray*}
I&=&R_{1\bar{1}1\bar{1}}R_{2\bar{2}1\bar{1}}+R_{1\bar{1}2\bar{2}}R_{2\bar{2}2\bar{2}}+2|R_{1\bar{1}1\bar{2}}|^2+|R_{1\bar{2}1\bar{2}}|^2 -|R_{1\bar{1}2\bar{2}}|^2\\
&\ge & R_{1\bar{1}2\bar{2}}\left(R_{1\bar{1}1\bar{1}}+R_{2\bar{2}2\bar{2}}-R_{1\bar{1}2\bar{2}}\right)\\
&=& R_{1\bar{1}2\bar{2}}\Scal -3 (R_{1\bar{1}2\bar{2}})^2 -R_{1\bar{1}2\bar{2}}\left(\sum_{j\ge 3} 2(R_{1\bar{1}j\bar{j}}+R_{2\bar{2}j\bar{j}})+\sum_{i, j\ge 3} R_{i\bar{i}j\bar{j}}\right).
\end{eqnarray*}
To get our estimate we only need to estimate $\sum R_{i\bar{i}i\bar{i}}$ from below. For $i\ne j$, let $E'_i=\frac{1}{\sqrt{2}}(E_i-E_j)$ and $E''_i=\frac{1}{\sqrt{2}}(E_i-E_j)$.  We have that
$$
4R_{E'\bar{E}'E''\bar{E}''}=R_{i\bar{i}i\bar{i}}+R_{j\bar{j}j\bar{j}}-R_{i\bar{j}i\bar{j}}-R_{j\bar{i}j\bar{i}}\ge 4 R_{1\bar{1}2\bar{2}}.
$$
Replacing $E_j$ by $\sqrt{-1}E_j$ we can get rid of the last two terms on the left hand side of  the above inequality and obtain that
$$
R_{i\bar{i}i\bar{i}}+R_{j\bar{j}j\bar{j}}\ge 4R_{1\bar{1}2\bar{2}}, \quad \mbox{ hence } \sum_{i\ge 3}R_{i\bar{i}i\bar{i}}\ge 2(n-2)R_{1\bar{1}2\bar{2}}.
$$
This implies the estimate (\ref{eq:6-key}) for $n\ge 4$.
\end{proof}

An alternative approach for $n\neq 3$ in the last part  of  argument for estimating $I$ (following Lemma 2.3 in \cite{BCW}) is as follows. Apply instead  the following two estimates :
\begin{eqnarray*}
R_{1\bar11\bar1}+R_{2\bar22\bar2}
&=&R_{1\bar1}+R_{2\bar2} -2R_{1\bar12\bar2} -\sum_{\a \ge 3}(R_{1\bar1\a \bar\a}+R_{2\bar2\a\bar\a} )\\
& \le & R_{1\bar1}+R_{2\bar2} -2(n-1)R_{1\bar12\bar2}; \\
R_{1\bar1}+R_{2\bar2} &\leq& \Scal -(n-2)(n+1)R_{1\bar12\bar2}, \text{  if } n\neq 3.
\end{eqnarray*}
The first one above is trivial. For the second one,  recall that $R^*=R+\ell \mbox{ } \tilde\I$ has NOB and $R^*_{1\bar12\bar2}=0$. Since NOB implies two-nonnegative Ricci (algebraically), we have for $n\neq 3$,
$R^*_{1\bar1}+R^*_{2\bar2} \le \Scal(R^*)$.
It then follows that
$$R_{1\bar1}+(n+1)\ell +R_{2\bar2} +(n+1)\ell  \le \Scal +n(n+1) \ell.$$

\section{Closed type-I ancient solutions}
In this section, we prove some classification results on closed Type I $\kappa$-noncollapsed ancient solutions, as consequences of the classification of shrinkers achieved in previous sections.
Recall that an ancient solution $(M, g(t))$ to the Ricci flow defined on $M \times (-\infty, 0)$ is called of type-I if there exists a constant $A$ such that $$|\Rm|(x,t) \le \frac{A}{|t|}.$$

We first give a complete classification of compact $\kappa$-noncollapsed Type I ancient solutions to the Ricci flow with strictly/weakly PIC$\mbox{}_1$, generalizing the second author's work \cite{Ni-type I}.

\begin{theorem}
Assume that $(M^n,g(t))$ is a compact type I, $\kappa$-noncollapsed (for some $\kappa>0$) ancient solution to the Ricci flow with (strictly) PIC$\mbox{}_1$. Then $(M, g(t))$ must be a quotient of $\mathbb{S}^n$.
\end{theorem}

\begin{proof}
We follow the argument in \cite{Ni-type I}.
Firstly, $(M,g(t))$ has nonnegative complex sectional curvature by Proposition \ref{prop:62}. This allows us to apply the blow-down procedure to $(M,g(t))$ as $t \to  -\infty$ using Proposition 11.2 of Perelman \cite{Perelman} and get an asymptotic shrinker $(M_\infty, g_\infty)$ with weakly PIC$\mbox{}_1$. Moreover, the same argument as Lemma 0.3 in \cite{Ni-type I} shows that
$(M_\infty, g_\infty)$ must be compact and topologically a quotient of $\mathbb{S}^n$, thus a metric quotient of $\mathbb{S}^n$ by Theorem \ref{thm:31}.
However, by \cite{Brendle}, we also have that $(M, g(t)) \to (M_\infty, g_\infty)$ as $t \to  0$. The fact that $(M,g(t))$ must be a shrinker follows from the equality case of the monotonicity of Perelman's entropy $\nu(M, g(t))$ as explained in \cite{Ni-type I}.
\end{proof}

The following corollary follows immediately from the strong maximum principle in \cite{BS2}.
\begin{corollary}
Assume that $(M^n,g(t))$ is a compact type I, $\kappa$-noncollapsed (for some $\kappa>0$) ancient solution to the Ricci flow with weakly PIC$\mbox{}_1$. Then $(M, g(t))$ must be quotients of products of symmetric spaces.
\end{corollary}

We also give a complete classification of compact $\kappa$-noncollapsed Type I ancient solutions to the K\"ahler-Ricci flow with $B^{\perp} \geq 0$.
\begin{theorem}
Assume that $(M^n,g(t))$ is a compact type I, $\kappa$-noncollapsed (for some $\kappa>0$) ancient solution to the K\"ahler-Ricci flow with $B^{\perp} >0$. Then $(M, g(t))$ must be, up to scaling, isometric to $\mathbb{P}^n$ with its Fubini-Study metric .
\end{theorem}

\begin{proof}

By Proposition \ref{prop:61}, we know that $(M,g(t))$ has nonnegative bisectional curvature. So we can apply the blow-down procedure to $(M,g(t))$ as $t \to  -\infty$ using Proposition 11.2 of Perelman \cite{Perelman} and its adaption to the K\"ahler case in \cite{Ni-MRL}, to get a limiting shrinker $(M_\infty, g_\infty)$ with $B^{\perp} \geq 0$.
By similarly arguments as Lemma 0.3 in \cite{Ni-type I}, we can conclude that $(M_\infty, g_\infty)$ must be compact. Thus it is forced to be topologically $\mathbb{P}^n$, thus isometric to $\mathbb{P}^n$ by Theorem \ref{thm:22}.

On the other hand, by the work of \cite{XChen, GuZhang, Wilking}, we also know that $(M, g(t)) \to (M_\infty, g_\infty)$ as $t \to  0$. The fact that $(M,g(t))$ must be a shrinker follows from the equality case of the monotonicity of Perelman's entropy $\nu(M, g(t))$ as illustrated in \cite{Ni-type I}.
\end{proof}

The strong maximum principle in \cite{BS2} and its extension in \cite{GuZhang, Wilking}
imply
\begin{corollary}
Assume that $(M^n,g(t))$ is a compact type I, $\kappa$-noncollapsed (for some $\kappa>0$) ancient solution to the K\"ahler-Ricci flow with $B^{\perp} \geq 0$. Then $(M, g(t))$ must be quotients of products of Hermitian symmetric spaces.
\end{corollary}

The examples in \cite{BKN} seem to suggest that the results no longer hold if we drop the assumption of the non-collapsing. However, since the example of \cite{BKN} is the Ricci flow of  Hermitian metrics, it remains interesting to construct examples of the K\"ahler-Ricci flow.

\section*{Acknowledgments} { We  thank Ovidiu Munteanu, Jiaping Wang and  Professor Hung-Hsi Wu for their interest to this work. We are also grateful to Burkhard Wilking for explaining \cite{BCW}}.

\bigskip

\bibliographystyle{plain}




\end{document}